\documentclass[a4paper]{amsart}
\usepackage{graphicx}
\usepackage{amssymb}
\usepackage{amsmath}
\usepackage{amsthm}
\usepackage{amscd}
\usepackage[all,2cell]{xy}

\UseAllTwocells \SilentMatrices
\newtheorem{thm}{Theorem}[section]

\newtheorem{cor}[thm]{Corollary}
\newtheorem{lem}[thm]{Lemma}
\newtheorem{exm}[thm]{Example}

\newtheorem{prop}[thm]{Proposition}
\theoremstyle{definition}

\theoremstyle{remark}

\numberwithin{equation}{section}

\begin{document}
\title[A note on the singularity category of an endomorphism ring]
{A note on the singularity category of an endomorphism ring}

\author[Xiao-Wu Chen] {Xiao-Wu Chen}

\thanks{The author is supported by National Natural Science Foundation of China (No.11201446) and NCET-12-0507.}
\date{\today}

\thanks{E-mail:
xwchen$\symbol{64}$mail.ustc.edu.cn}

\maketitle

\dedicatory{}%
\commby{}%

\begin{abstract}
We propose the notion of partial resolution of a ring, which is by definition the endomorphism ring of a certain generator of the given ring. We prove that the singularity category of the partial resolution is a quotient of the singularity category of the given ring. Consequences and examples are given.
\end{abstract}

\section{Introduction}
Let $A$ be a left coherent ring with a unit. Denote by $A\mbox{-mod}$ the category of finitely presented left $A$-modules and by $\mathbf{D}^b(A\mbox{-mod})$ the bounded derived category. Following \cite{Buc, Or04}, the \emph{singularity category}  $\mathbf{D}_{\rm sg}(A)$ of $A$ is the Verdier quotient category of  $\mathbf{D}^b(A\mbox{-mod})$ with respect to the subcategory formed by perfect complexes. We denote by $q\colon \mathbf{D}^b(A\mbox{-mod})\rightarrow \mathbf{D}_{\rm sg}(A)$ the quotient functor.  The singularity category measures the homological singularity of $A$.

Let $_AM$ be a finitely presented $A$-module. Denote by ${\rm add}\; M$ the full subcategory consisting of direct summands of finite direct sums of $M$. A \emph{finite $M$-resolution} of an $A$-module $X$ means an exact sequence $0\rightarrow N^{-n}\rightarrow N^{1-n}\rightarrow \cdots \rightarrow N^{-1}\rightarrow N^0\rightarrow X$ with each $N^{-i}\in {\rm add}\; M$ which remains exact after applying the functor ${\rm Hom}_A(M, -)$.

Recall that an $A$-module $M$ is a \emph{generator} if $A$ lies in ${\rm add}\; M$. We consider the opposite ring of its endomorphism ring $\Gamma={\rm End}_A(M)^{\rm op}$, and $M$ becomes an $A$-$\Gamma$-bimodule. In particular,  if $\Gamma$ is left coherent, we have the functor $M\otimes_\Gamma- \colon \Gamma\mbox{-mod}\rightarrow A\mbox{-mod}$.

Let $M$ be a generator with $\Gamma={\rm End}_A(M)^{\rm op}$. Following the idea of \cite{Vandb}, we call $\Gamma$ a \emph{partial resolution} of $A$ if $\Gamma$ is left coherent and  any $A$-module $X$ has a finite $M$-resolution, provided that it fits into an exact sequence $0\rightarrow X\rightarrow N^1\rightarrow N^2\rightarrow 0$ with $N^i\in {\rm add}\; M$. We mention that similar idea might trace back to \cite{Au}.

 The following result justifies the terminology: the partial resolution $\Gamma$ has ``better" singularity than $A$. The result is inspired by \cite[Theorem 1.2]{KZ}, where the case that $A$ is an artin algebra with finitely many indecomposable Gorenstein projective modules is studied.

 \begin{prop}\label{prop:1}
 Let $A$ be a left coherent ring, and let $\Gamma$ be a partial resolution of $A$ as above. Then there is a triangle equivalence
 $$\mathbf{D}_{\rm sg}(\Gamma) \stackrel{\sim}\longrightarrow \mathbf{D}_{\rm sg}(A)/\langle q(M)\rangle  $$
induced by the functor $M\otimes_\Gamma-$.
 \end{prop}

 Here, we identify the $A$-module $M$ as the stalk complex concentrated at degree zero, and then $q(M)$ denotes the image in
 $\mathbf{D}_{\rm sg}(A)$. We denote by $\langle q(M)\rangle$ the smallest triangulated subcategory of $\mathbf{D}_{\rm sg}(A)$ that contains $q(M)$ and is closed under taking direct summands. Then $\mathbf{D}_{\rm sg}(A)/\langle q(M) \rangle$ is the corresponding Verdier quotient category.

 The aim of this note is to prove Proposition \ref{prop:1} and discuss related results and examples for artin algebras.

 \section{The proof of Proposition \ref{prop:1}}

 To give the proof, we collect in the following lemma some well-known facts. Throughout, $A$ is a left coherent ring and $_AM$ is a generator with $\Gamma={\rm End}_A(M)^{\rm op}$ left coherent.

 Recall the functor $M\otimes_\Gamma -\colon \Gamma\mbox{-mod}\rightarrow A\mbox{-mod}$. Denote by $\mathcal{N}$ the essential kernel of $M\otimes_\Gamma -$, that is, the full subcategory of $\Gamma\mbox{-mod}$ consisting of $_\Gamma Y$ such that $M\otimes_\Gamma Y\simeq 0$.

 We will also consider the category $A\mbox{-Mod}$ of arbitrary left $A$-modules. Note that the functor $M\otimes_\Gamma -\colon \Gamma\mbox{-Mod}\rightarrow A\mbox{-Mod}$ is left adjoint to  ${\rm Hom}_A(M, -)\colon A\mbox{-Mod}\rightarrow \Gamma\mbox{-Mod}$. Denote by $\mathcal{N}'$ the essential kernel of $M\otimes_\Gamma -\colon \Gamma\mbox{-Mod}\rightarrow A\mbox{-Mod}$. The functor ${\rm Hom}_A(M, -)$ induces an equivalence ${\rm add}\; M\simeq \Gamma\mbox{-proj}$, where $\Gamma\mbox{-proj}$ denotes the category of finitely generated projective $\Gamma$-modules.

For a class $\mathcal{S}$ of objects in a triangulated category $\mathcal{T}$, we denote by $\langle \mathcal{S}\rangle$ the smallest triangulated subcategory that contains $\mathcal{S}$ and is closed under taking direct summands. For example, the subcategory  of $\mathbf{D}^b(A\mbox{-mod})$ formed by perfect complexes equals $\langle A\rangle$; here, we view a module as a stalk complex concentrated on degree zero.

 \begin{lem}\label{lem:1}
 Let the $A$-module $M$ be a generator and $\Gamma={\rm End}_A(M)^{\rm op}$. Then the following results hold.
 \begin{enumerate}
 \item the functor ${\rm Hom}_A(M, -)\colon A\mbox{-{\rm Mod}}\rightarrow \Gamma\mbox{-{\rm Mod}}$ is fully faithful;
 \item the right $\Gamma$-module $M_\Gamma$ is projective, and then the subcategory $\mathcal{N}$ of $\Gamma\mbox{-{\rm mod}}$ is a Serre subcategory, that is, it is closed under submodules, factor modules and extensions;
 \item the functor   $M\otimes_\Gamma -\colon \Gamma\mbox{-{\rm mod}}\rightarrow A\mbox{-{\rm mod}}$ induces an equivalence $\Gamma\mbox{-{\rm mod}}/\mathcal{N}\stackrel{\sim}\longrightarrow A\mbox{-{\rm mod}}$, where $\Gamma\mbox{-{\rm mod}}/\mathcal{N}$ denotes the quotient category of $\Gamma\mbox{-{\rm mod}}$ with respect to the Serre subcategory $\mathcal{N}$;
     \item the functor   $M\otimes_\Gamma -\colon \Gamma\mbox{-{\rm mod}}\rightarrow A\mbox{-{\rm mod}}$ induces a triangle equivalence \;  $\mathbf{D}^b(\Gamma\mbox{-{\rm mod}})/\langle \mathcal{N}\rangle \stackrel{\sim}\longrightarrow \mathbf{D}^b(A\mbox{-{\rm mod}})$.
 \end{enumerate}
 \end{lem}

\begin{proof}
(1) is contained in \cite[Theorem X.4.1(i)]{St}, and (2) is contained in \cite[Proposition IV. 6.7(i)]{St}. In particular, the functor $M\otimes_\Gamma-$ is exact. Then  we recall that the essential kernel of any exact functor between abelian categories is a Serre subcategory. We infer from (1) an equivalence $\Gamma\mbox{-{\rm Mod}}/\mathcal{N}'\stackrel{\sim}\longrightarrow A\mbox{-{\rm Mod}}$; consult \cite[Proposition I.1.3 and Subsection 2.5 d)]{GZ}. Then (3) follows from \cite[Proposition A.5]{Kra01}. Thanks to (3), the last statement follows from a general result \cite[Theorem 3.2]{Mi}.
\end{proof}

 The argument in the proof of the following result is essentially contained in  the proof of \cite[Chapter III, Section 3, Theorem]{Au}.

 \begin{lem}\label{lem:2}
 Keep the notation as above. Then any $\Gamma$-module in $\mathcal{N}$ has finite projective dimension if and only if $\Gamma$ is a partial resolution of $A$.
 \end{lem}

 \begin{proof}
 For the ``if" part, assume that $\Gamma$ is a partial resolution of $A$, and let $_\Gamma Y$ be a module in $\mathcal{N}$, that is, $M\otimes_\Gamma Y\simeq 0$. Take an exact sequence $0\rightarrow Y'\rightarrow P^{1}\stackrel{f}\rightarrow P^2\rightarrow Y\rightarrow 0$ in $\Gamma\mbox{-mod}$ such that each $P^i$ is projective. Recall the equivalence ${\rm Hom}_A(M, -)\colon {\rm add}\; M\rightarrow \Gamma\mbox{-proj}$. Then there is a map $g\colon N^{1}\rightarrow N^2$ with $N^i\in {\rm add}\; M$ such that ${\rm Hom}_A(M, g)$ is identified with $f$. Then $M\otimes_\Gamma Y\simeq 0$ implies that $g$ is epic. Moreover, if $X={\rm Ker}\; g$, then $Y'\simeq {\rm Hom}_A(M, X)$. Then by assumption $X$ admits a finite $M$-resolution $0\rightarrow N^{-n}\rightarrow N^{1-n}\rightarrow \cdots \rightarrow N^{-1}\rightarrow N^0\rightarrow X\rightarrow 0$ with each $N^{-i}\in {\rm add}\; M$. Applying ${\rm Hom}_A(M, -)$ to it, we obtain a finite projective resolution of $_\Gamma Y'$. In particular, $_\Gamma Y$ has finite projective dimension. The ``only if" part follows by reversing the argument.
 \end{proof}

Recall that $\langle A\rangle$ in $\mathbf{D}^b(A\mbox{-mod})$ equals the subcategory formed by perfect complexes. The singularity category is given by $\mathbf{D}_{\rm sg}(A)=\mathbf{D}^b(A\mbox{-mod})/\langle A \rangle$. We denote by $q\colon \mathbf{D}^b(A\mbox{-mod})\rightarrow \mathbf{D}_{\rm sg}(A)$ the quotient functor.

The following observation is due to \cite[Proposition 3.3]{KY} in a slightly different setting. We include the proof for completeness.

\begin{lem}\label{lem:3}
Keep the notation as above. Then the functor $M\otimes_\Gamma -\colon \Gamma\mbox{-{\rm mod}}\rightarrow A\mbox{-{\rm mod}}$ induces a triangle $\mathbf{D}_{\rm sg}(\Gamma)/\langle q(\mathcal{N})\rangle \stackrel{\sim}\longrightarrow \mathbf{D}_{\rm sg}(A)/\langle q(M)\rangle$.
\end{lem}

\begin{proof}
Recall from Lemma \ref{lem:1}(4) the triangle equivalence $\mathbf{D}^b(\Gamma\mbox{-{\rm mod}})/\langle \mathcal{N}\rangle \stackrel{\sim}\longrightarrow \mathbf{D}^b(A\mbox{-{\rm mod}})$ induced by $M\otimes_\Gamma -$. In particular, it sends $\Gamma$ to $M$. Hence, it induces a triangle equivalence $$\frac{\mathbf{D}^b(\Gamma\mbox{-{\rm mod}})/\langle \mathcal{N}\rangle}{\langle \Gamma \rangle} \stackrel{\sim}\longrightarrow \mathbf{D}^b(A\mbox{-{\rm mod}})/\langle M\rangle.$$
Since $_AM$ is a generator, we have $\langle A\rangle\subseteq \langle M\rangle \subseteq \mathbf{D}^b(A\mbox{-{\rm mod}})$. It follows from \cite[\S2, 4-3 Corollaire]{Ver77} that $\mathbf{D}^b(A\mbox{-{\rm mod}})/\langle M\rangle=\mathbf{D}_{\rm sg}(\Gamma)/\langle q(M)\rangle$. For the same reason,  $ (\mathbf{D}^b(\Gamma\mbox{-{\rm mod}})/\langle \mathcal{N}\rangle)/{\langle \Gamma \rangle}$ is identified with $\mathbf{D}_{\rm sg}(\Gamma)/\langle q(\mathcal{N})\rangle $. Then we are done.
\end{proof}

\vskip 5pt

\noindent {\bf Proof of Proposition \ref{prop:1}}.\quad Recall that for an $A$-module $X$, $q(X)\simeq 0$ in $\mathbf{D}_{\rm sg}(A)$ if and only if $X$ has finite projective dimension. Hence by Lemma \ref{lem:2}, all objects in $q(\mathcal{N})$ are isomorphic to zero. Then the result follows from Lemma \ref{lem:3}. \hfill $\square$

\vskip 10pt

We observe the following immediate consequence.

\begin{cor}
Let $A$ be a left coherent ring and $_AM$ a generator. Assume that $\Gamma={\rm End}_A(M)^{\rm op}$ is left coherent such that each finitely presented $\Gamma$-module has finite projective dimension. Then $\Gamma$ is a partial resolution of $A$ and $\mathbf{D}_{\rm sg}(A)=\langle q(M)\rangle$.
\end{cor}

In the situation of this corollary, we might even call $\Gamma$ a \emph{resolution} of $A$. Such a resolution always exists for any left artinian ring; see \cite[Chapter III, Section 3, Theorem]{Au}.

\begin{proof}
The first statement follows from Lemma \ref{lem:2}. By assumption, $\mathbf{D}_{\rm sg}(\Gamma)=0$ and thus the second statement follows from Proposition \ref{prop:1}.
\end{proof}

\section{Consequences and Examples}

We draw some consequences of Proposition \ref{prop:1} for artin algebras. In this section, $A$ will be an artin algebra over a commutative artinian ring $R$.

\subsection{} Recall that an $A$-module $M$ is \emph{Gorenstein projective} provided that $M$ is reflexive, ${\rm Ext}_A^i(M, A)=0={\rm Ext}_{A^{\rm op}}^i(A, M^*)$ for all $i\geq 1$. Here, $M^*={\rm Hom}_A(M, A)$. Any projective module is Gorenstein projective.

Denote by $A\mbox{-Gproj}$ the full subcategory of $A\mbox{-mod}$ consisting of Gorenstein projective modules. It is closed under extensions and kernels of surjective maps. In particular, $A\mbox{-Gproj}$ is a Frobenius exact category whose projective-injective objects are precisely projective $A$-modules. Denote by $A\mbox{-\underline{Gproj}}$ the stable category; it is naturally triangulated by \cite[Theorem I.2.8]{Hap88}. The Hom spaces in the stable category are denoted by $\underline{\rm Hom}$.

Recall from \cite[Theorem 1.4.4]{Buc} or  \cite[Theorem 4.6]{Hap91} that there is a  full triangle embedding $F_A\colon A\mbox{-\underline{Gproj}}\rightarrow \mathbf{D}_{\rm sg}(A)$ sending $M$ to $q(M)$. Moreover, $F_A$ is dense, thus a triangle equivalence if and only if $A$ is \emph{Gorenstein}, that is, the injective dimension of the regular $A$-module has finite injective dimension on both sides.

A full subcategory $\mathcal{C}\subseteq A\mbox{-Gproj}$ is \emph{thick} if it contains  all projective modules and for any short exact sequence $0\rightarrow M_1\rightarrow M_2\rightarrow M_3\rightarrow 0$ with terms in $A\mbox{-Gproj}$, all $M_i$ lie in $\mathcal{C}$ provided that two of them lie in $\mathcal{C}$. In this case, the stable category $\underline{\mathcal{C}}$ is a triangulated subcategory of $A\mbox{-\underline{Gproj}}$, and thus via $F_A$, a triangulated subcategory of $\mathbf{D}_{\rm sg}(A)$.

\begin{prop}\label{cor:1}
Let $_AM$ be a generator which is Gorenstein projective such that ${\rm add}\; M$ is a thick subcategory of $A\mbox{-{\rm Gproj}}$. Then $\Gamma={\rm End}_A(M)^{\rm op}$ is a partial resolution of $A$ and thus we have  a triangle equivalence
$$\mathbf{D}_{\rm sg}(\Gamma) \stackrel{\sim}\longrightarrow \mathbf{D}_{\rm sg}(A)/{\underline{{\rm add}}\; M},$$
which is further triangle equivalent to $q(M)^\perp$.
\end{prop}

Here, $q(M)^\perp$ denotes the full subcategory of $\mathbf{D}_{\rm sg}(A)$ consisting of objects $X$ with ${\rm Hom}_{\mathbf{D}_{\rm sg}(A)}(q(M), X)=0$. Since ${\rm add}\; M$ is a thick subcategory of  $A\mbox{-{\rm Gproj}}$, it follows that $q(M)^\perp$ is a triangulated subcategory of $\mathbf{D}_{\rm sg}(A)$.

\begin{proof}
For any exact sequence $0\rightarrow X\rightarrow N^1\rightarrow N^2\rightarrow 0$ with $N^i\in {\rm add}\; M$, we have that $X$ is Gorenstein projective, since $A\mbox{-Gproj}$ is closed under kernels of surjective maps. Then by assumption, $X$ lies in ${\rm add}\; M$, in particular, $X$ admits a finite $M$-resolution. Then $\Gamma$ is a partial resolution of $A$. We note that $\langle q(M)\rangle = \underline{{\rm add}}\; M$. Then the first equivalence follows from Proposition \ref{prop:1} immediately.

 For the second one, we note by \cite[Proposition 1.21]{Or04} that for each object $X\in \mathbf{D}_{\rm sg}(A)$, ${\rm Hom}_{\mathbf{D}_{\rm sg}(A)}(q(M), X)$ is a finite length $R$-module and then the cohomological functor ${\rm Hom}_{\mathbf{D}_{\rm sg}(A)}(q(M), -)\colon \underline{\rm add}\; M\rightarrow R\mbox{-mod}$ is representable. Hence, the triangulated subcategory $\underline{\rm add}\; M$ of $\mathbf{D}_{\rm sg}(A)$ is right admissible in the sense of \cite{BK}. Then the result follows from \cite[Proposition 1.6]{BK}.
\end{proof}

A special case is of independent interest: an algebra $A$ is \emph{CM-finite} provided that there exists a module $M$ such that $A\mbox{-Gproj}={\rm add}\; M$. In this case,  we obtain a triangle equivalence
$$\mathbf{D}_{\rm sg}(\Gamma) \stackrel{\sim}\longrightarrow \mathbf{D}_{\rm sg}(A)/{A\mbox{-\underline{Gproj}}}.$$
In particular,  $\Gamma$ has finite global dimension if and
only if $A$ is Gorenstein. This triangle equivalence is due to \cite[Theorem 1.2]{KZ}.

For the following example, we recall that any semisimple abeian category $\mathcal{A}$ has a unique (trivial) triangulated structure with the translation functor given by any auto-equivalence $\Sigma$. This triangulated category is denoted by $(\mathcal{A}, \Sigma)$.

\begin{exm}
Let $k$ be a field. Consider the algebra $A$ given by the following quiver with relations $\{\beta\alpha\gamma \beta \alpha, \alpha \gamma \beta \alpha \gamma \beta\}$
\[\xymatrix{
1  \ar[r]^\alpha & 2\ar[r]^\beta &  3. \ar@/^0.7pc/[ll]^\gamma
}\]
Here, we write the concatenation of arrows from right to left. The algebra $A$ is Nakayama with admissible sequence $(5, 6, 6)$.

For each vertex $i$, we associate the simple $A$-module $S_i$. Consider the unique indecomposable module $S_2^{[3]}$ with top $S_2$ and length $3$; it is the unique indecomposable non-projective Gorenstein projective $A$-module; see \cite[Proposition 3.14(3)]{CYe} or \cite{Rin}. Set $M=A\oplus S_2^{[3]}$. Then ${\rm add}\; M=A\mbox{-{\rm Gproj}}$. The algebra $\Gamma={\rm End}_A(M)^{\rm op}$ is given by the following quiver with relations $\{ab, \beta b a \alpha, ba-\alpha \gamma \beta\}$
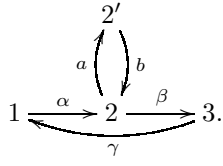
\[\xymatrix{
& 2'\ar@/^0.5pc/[d]^b \\
1  \ar[r]^\alpha & 2 \ar@/^0.5pc/[u]^a \ar[r]^\beta &  3. \ar@/^0.7pc/[ll]^\gamma
}\]
Then by Proposition \ref{cor:1} we have triangle equivalences
$$\mathbf{D}_{\rm sg}(\Gamma) \stackrel{\sim}\longrightarrow \mathbf{D}_{\rm sg}(A)/{A\mbox{-\underline{\rm Gproj}}} \stackrel{\sim}\longrightarrow q(S_2^{[3]})^\perp.$$
Here, we recall that $q(S_2^{[3]})^\perp=q(M)^\perp$ is the triangulated subcategory of $\mathbf{D}_{\rm sg}(A)$ consisting of objects $X$ with ${\rm Hom}_{\mathbf{D}_{\rm sg}(A)}(q(S_2^{[3]}), X)=0$.  The singularity category $\mathbf{D}_{\rm sg}(A)$ is triangle equivalent to the stable category $A'\mbox{-\underline{\rm mod}}$ for an elementary connected Nakayama algebra $A'$ with admissible sequence $(4,4)$ such that $q(S_2^{[3]})$ corresponds to an $A'$-module of length $2$; see \cite[Corollary 3.11]{CYe}.

 Then explicit calculation  in $A'\mbox{-\underline{\rm mod}}$ yields that $q(S_2^{[3]})^\perp$ is equivalent to $k\times k\mbox{-{\rm mod}}$; moreover, the translation functor $\Sigma$ is induced by the algebra automorphism of $k\times k$ that switches the coordinates. In summary, we obtain a triangle equivalence
$$\mathbf{D}_{\rm sg}(\Gamma) \stackrel{\sim}\longrightarrow (k\times k\mbox{-{\rm mod}}, \Sigma).$$
We mention that by \cite[Theorem 1.1]{KZ} any Gorenstein projective $\Gamma$-module is projective. In particular, the algebra $\Gamma$ is not Gorenstein.
\end{exm}

\subsection{} Let $k$ be a field and $Q$ a finite quiver without oriented cycles. Consider the path algebra $kQ$ and the algebra  $A=kQ[\epsilon]=kQ\otimes_k k[\epsilon]$ of dual numbers with coefficients in $kQ$, where $k[\epsilon]$ is the algebra of dual numbers. Then $A$ is Gorenstein. The stable category $A\mbox{-\underline{\rm Gproj}}$, which is equivalent to $\mathbf{D}_{\rm sg}(A)$, is studied in \cite{RZ}; compare \cite[Section 5]{CY}.

For an $A$-module $Y$, $\epsilon$ induces a $kQ$-module map $\epsilon_Y\colon Y\rightarrow Y$ satisfying $\epsilon_Y^2=0$. The \emph{cohomology} of $Y$ is defined as $H(Y)={{\rm Ker} \; \epsilon_Y}/{{\rm Im}\; \epsilon_Y}$.
This gives rise to a  functor
$$H\colon A\mbox{-{\rm Gproj}}\longrightarrow kQ\mbox{-mod},$$
which induces a cohomological functor $A\mbox{-\underline{\rm Gproj}}\rightarrow kQ\mbox{-mod}$. It follows that for each subcategory $\mathcal{C}$ of $kQ\mbox{-mod}$ which is closed under extensions, kernels of surjective maps and cokernels of injective maps, the corresponding subcategory $H^{-1}(\mathcal{C})$ of $A\mbox{-Gproj}$ is thick.

Recall that the path algebra $kQ$ is hereditary. For each $kQ$-module $X$, consider its minimal projective $kQ$-resolution $0\rightarrow P^{-1}\stackrel{d}\rightarrow P^0\rightarrow X\rightarrow 0$. Set $\eta(X)=P^{-1}\oplus P^0$; it is an $A$-module such that $\epsilon$ acts on $P^{-1}$ by $d$ and on $P^0$ as zero. If $X$ is indecomposable, the $A$-module $\eta(X)$ is Gorenstein projective which is indecomposable and non-projective. Observe that $H(\eta(X))\simeq X$. Indeed, if $Y$ is an indecomposable non-projective Gorenstein projective $A$-module satisfying $H(Y)\simeq X$, then $Y$ is isomorphic to $\eta(X)$.  For two $kQ$-modules $X$ and $X'$, we have a natural isomorphism
\begin{align}\label{equ:1}
\underline{\rm Hom}_A(\eta(X), \eta(X'))\simeq {\rm Hom}_{kQ}(X, X')\oplus {\rm Ext}^1_{kQ}(X, X').
\end{align}
We refer the details to \cite[Theorems 1 and 2]{RZ}.

Let $E$ be an \emph{exceptional} $kQ$-module, that is, an indecomposable $kQ$-module such that ${\rm Ext}^1_{kQ}(E, E)=0$. It is well known that the full subcategory $E^{\perp_{0,1}}$ of $kQ\mbox{-mod}$ consisting of modules $X$ with ${\rm Hom}_{kQ}(E, X)=0={\rm Ext}^1_{kQ}(E, X)$ is equivalent to $kQ'\mbox{-mod}$ for some quiver $Q'$; moreover, the number of vertices of $Q'$ is less than the number of vertices of $Q$ by one; see \cite[Theorem 2.3]{Sch}.

 \begin{prop}\label{prop:2}
 Keep the notation as above. Set $M=A\oplus \eta(E)$ and $\Gamma={\rm End}_A(M)^{\rm op}$. The following statements hold.
 \begin{enumerate}
 \item ${\rm add}\; M$ is a thick subcategory of $A\mbox{-{\rm Gproj}}$, and the corresponding stable category $\underline{\rm add} \;M$ is triangle equivalent to $(k\mbox{-{\rm mod}}, {\rm Id}_{k\mbox{-}{\rm mod}})$;
 \item $q(M)^\perp$ is triangle equivalent to $kQ'[\epsilon]\mbox{-\underline{\rm Gproj}}$;
 \item there is a triangle equivalence $\mathbf{D}_{\rm sg}(\Gamma) \stackrel{\sim}\longrightarrow kQ'[\epsilon]\mbox{-\underline{\rm Gproj}}$.
 \end{enumerate}

 \end{prop}

\begin{proof}
For (1), recall that ${\rm End}_{kQ}(E)\simeq k$. It follows that the subcategory ${\rm add}\; E$ of $kQ\mbox{-mod}$ is  closed under extensions, kernels of surjective maps and cokernels of injective maps. Observe that ${\rm add}\; M=H^{-1}({\rm add}\; E)$. Then the thickness of  ${\rm add}\; M$ follows. The stable category $\underline{\rm add} \;M$ is given by ${\rm add}\; \eta(E)$; moreover, by (\ref{equ:1}) we have $\underline{\rm End}_{A}(\eta(E))\simeq k$. We observe that the translation functor acts on ${\rm add}\; \eta(E)$ as the identity; see \cite[Proposition 4.10]{RZ}. Then we infer (1).

We identify $\mathbf{D}_{\rm sg}(A)$ with $A\mbox{-}\underline{\rm Gproj}$. By (\ref{equ:1}), $q(M)^\perp=q(\eta(E))^\perp$ equals the stable category of $H^{-1}(E^{\perp_{0, 1}})$. Since $E^{\perp_{0, 1}}$ is equivalent to $kQ'\mbox{-mod}$, it follows from \cite[Theorem 1]{RZ} that the stable category of  $H^{-1}(E^{\perp_{0, 1}})$ is triangle equivalent to $kQ'[\epsilon]\mbox{-\underline{\rm Gproj}}$. This equivalence might also be deduced from the results in \cite[Sections 3 and 5]{CY}.

The last statement follows from (2) and Proposition \ref{cor:1}.
\end{proof}

We conclude with an example of Proposition \ref{prop:2}.

\begin{exm}
Let $Q$ be the quiver $1 \stackrel{\alpha}\longrightarrow 2$ and let $A=kQ[\epsilon]$. Denote by $S_1$ the simple $kQ$-module corresponding to the vertex $1$; it is an exceptional $kQ$-module. Then $\eta(S_1)=kQ$  such that $\epsilon$ acts as the multiplication of $\alpha$ from the right. The algebra $\Gamma={\rm End}_A(A\oplus \eta(S_1))^{\rm op}$ is given by the following quiver with relations $\{\beta\alpha, \alpha\delta-\gamma\beta\}$
\[\xymatrix{
1 \ar@<+1.0ex>[r]^-{\alpha} & 2 \ar@<+1.0ex>[l]^-{\delta} \ar@<+1.0ex>[r]^-{\beta} & 3.\ar@<+1.0ex>[l]^-{\gamma}
}
\]
The corresponding quiver $Q'$ in Proposition \ref{prop:2} is a single vertex, and then the stable category $kQ'[\epsilon]\mbox{-\underline{\rm Gproj}}$ is triangle equivalent to $(k\mbox{-{\rm mod}}, {\rm Id}_{k\mbox{-}{\rm mod}})$. Hence, by Proposition \ref{prop:2}(3), we obtain a triangle equivalence
$$\mathbf{D}_{\rm sg}(\Gamma) \stackrel{\sim}\longrightarrow  (k\mbox{-{\rm mod}}, {\rm Id}_{k\mbox{-}{\rm mod}}).$$

We mention that the simple $\Gamma$-module corresponding to the vertex $3$ is localizable, whose corresponding left retraction $L(\Gamma)$ is an elementary connected Nakayama algebra with admissible sequence $(3, 4)$; see \cite{CYe}. Then by \cite[Lemma 3.12(2) and Proposition 2.6]{CYe} any Gorenstein projective $\Gamma$-module is projective; in particular, $\Gamma$ is not Gorenstein. The above triangle equivalence might also be deduced from \cite[Proposition 2.13 and Corollary 3.11]{CYe}.
\end{exm}

\vskip 5pt

\noindent {\bf Acknowledgements}\; The author thanks Professor Pu Zhang for his  talk given in USTC, which inspires this note.

\bibliography{}

\vskip 10pt

 {\footnotesize \noindent Xiao-Wu Chen, School of Mathematical Sciences,
  University of Science and Technology of
China, Hefei 230026, Anhui, PR China \\
Wu Wen-Tsun Key Laboratory of Mathematics, USTC, Chinese Academy of Sciences, Hefei 230026, Anhui, PR China.\\
URL: http://home.ustc.edu.cn/$^\sim$xwchen}

\end{document}